\documentclass[
12pt,oneside,english,a4paper]{amsart}
\usepackage[T1]{fontenc}
\usepackage[utf8]{inputenc}
\usepackage[dvipsnames]{xcolor}
\usepackage[yyyymmdd,hhmmss]{datetime}
\usepackage[numbers,sort,compress]{natbib}
\usepackage{amstext,amsthm,amssymb}
\usepackage{xparse,mathrsfs,mathtools}
\usepackage[english]{babel}
\usepackage[pdftex,unicode=true,%
	pdfusetitle,pdfdisplaydoctitle=true,%
	pdfpagemode=UseOutlines,%
	bookmarks=true,bookmarksnumbered=false,bookmarksopen=true,bookmarksopenlevel=2,%
	breaklinks=true,%
	pdfencoding=unicode,psdextra,
	pdfcreator={},
	pdfborder={0 0 0}
]{hyperref}
\usepackage{lmodern,microtype}
\usepackage{geometry}
\usepackage[capitalise,nameinlink]{cleveref}
\usepackage[cal=zapfc,bb=ams,frak=esstix,scr=boondoxo]{mathalfa}

\newtheorem*{thm*}{Theorem}
\newtheorem{thm}{Theorem}[section]
\newtheorem{lem}[thm]{Lemma}
\newtheorem{cor}[thm]{Corollary}

\theoremstyle{remark}

\newtheorem*{rem*}{Remark}

\NewDocumentCommand\e{ s O{} m }{%
	\IfBooleanTF{#1}{%
		\operatorname{e}_{#2}\parentheses*{#3}%
	}{\operatorname{e}_{#2}\parentheses{#3}}%
}
\newcommand{\meas}{\operatorname{meas}}

\DeclarePairedDelimiter\parentheses{\lparen}{\rparen}
\DeclarePairedDelimiter\braces{\lbrace}{\rbrace}

\NewDocumentCommand\set{ s o m o }{%
	\IfBooleanTF{#1}{\IfNoValueTF{#4}{\braces*{#3}}{\braces*{\,#3:#4\,}}}{%
	\IfNoValueTF{#2}{\IfNoValueTF{#4}{\braces{#3}}{\braces{\,#3:#4\,}}}{%
	\IfNoValueTF{#4}{\braces[#2]{#3}}{\braces[#2]{\,#3:#4\,}}}}%
}

\crefformat{equation}{#2(#1)#3}
\crefformat{enumi}{#2(#1)#3}
\numberwithin{equation}{section}
\usepackage{xifthen,ifthen}
\makeatletter
\newcounter{@ToDo}
\newcommand{\todo@helper}[1]{%
	({\color{blue}TODO~\arabic{@ToDo}: {#1\@addpunct{.}}})%
}
\newcommand{\todo}[1]{\stepcounter{@ToDo}%
	\relax\ifmmode\text{\todo@helper{#1}}%
	\else\todo@helper{#1}\fi%
}
\makeatother

\title{On small fractional parts of polynomial-like functions}
\date{\today{}}

\author{Paolo~Minelli}
\address{
	Paolo~Minelli\\%
	Institut für Analysis und Zahlentheorie\\%
	TU~Graz\\
	Kopernikusgasse~24/II\\%
	8010~Graz\\%
	Austria}
\email{minelli@math.tugraz.at}
\begin{document}
\begin{abstract}
In a recent paper, Madritsch and Tichy established Diophantine inequalities for the fractional parts of polynomial-like functions. In particular, for $f(x)=x^k+x^c$ where $k$ is a positive integer and $c>1$ is a non-integer, and any fixed $\xi\in [0,1]$ they obtained 
\[\min_{2\leq p\leq X} \Vert \xi \lfloor f(p)\rfloor \Vert\ll_{k,c,\epsilon} X^{-\rho_1(c,k)+\epsilon}\]
for $\rho_1(c,k)>0$ explicitly given. In the present note, we improve upon their results in the case $c>k$ and $c>4$.
\end{abstract}
\maketitle

\section{Introduction and statement of results}
Vinogradov (1927) proved, answering a question stated by Hardy and Littlewood in the previous decade, that for any real number $\theta$ and $k\geq 2$ integer, one has
\begin{equation}\label{thm:Vinogradov:originalresult}
\min_{1\leq n\leq X} \Vert n^k \theta\Vert \ll X^{-\eta_1(k)+\epsilon},
\end{equation}

with the implied constant depending only on $k$ and $\epsilon$. This result was subsequently refined by Heilbronn \cite{Heilbronn1948} and several other authors. In particular, for small $k$, we mention the results of Zaharescu \cite{Zaharescu1995}, who obtained \cref{thm:Vinogradov:originalresult} with $\eta_1(2)=4/7$, which is, to the best of our knowledge, the actual record for $k=2$. For $k\geq 11$, the best known exponent for \cref{thm:Vinogradov:originalresult} is provided by Wooley \cite{Wooley1993}. We refer to Baker \cite{BakerIneq} and to the more up to date article of the same author \cite{Baker2016} for a comprehensive exposition of known results regarding \cref{thm:Vinogradov:originalresult}.
\\
Let now $k\geq 1$ and $f(x)=\sum_{j=1}^k a_kx^k$ be a polynomial with   at least one irrational coefficient The problem of establishing results of the form
\begin{equation}\label{prob:Davenport}
\min_{1\leq n\leq X} \Vert f(n)\Vert\ll X^{-\eta_2(k)+\epsilon},
\end{equation}

for $\eta_2(k)>0$ (conjecturally $\eta_2=1$) and implied constant depending only upon $k$ and $\epsilon$ was posed first by Davenport (1967). To the best of our knowledge, for a generic polynomial $f$, the best known results are due to Baker \cite{Baker2016}, \cite{Baker1982}, who proved \cref{prob:Davenport} for the exponents $\eta_2(k)=1/2k(k-1)$ (for $k\geq 8$) resp. $\eta_2(k)=2^{1-k}$ (for $2\leq k\leq 7$). An adjacent problem is to consider small fractional parts of polynomials over sparser sets. The case of primes obtained particular attention. The reader is refereed to the papers of Harman \cite{HarmanTrigI} and recent improvements due to Baker \cite{Bakerprimes2017}, \cite{Bakerprimes2018}. 
\\
\\
Now, in the present note, a \textit{pseudo polynomial} is a function $f:\mathbb{R}\to \mathbb{R}$ of the form
\begin{align}\label{pseudopolydef}
f(x)=\sum_{j=1}^d \alpha_jx^{\theta_j}
\end{align}
for $\alpha_j$ positive reals and $1\leq \theta_1<\theta_2<\dots\theta_d$, with at least one non integral $\theta_j$, $1\leq j\leq d$. We may split such a function into a \textit{polynomial part}, which we denote by $P$, and a part of the shape (\ref{pseudopolydef}), where \textit{all} the exponents $\theta_j$ are non integral. This second part will be called the \textit{pseudo-part} and denoted by $\phi$. Defining the degree of the pseudo part to be the largest exponent appearing in the representation and denoting this by $\deg(\phi)$ we may define (abusing notation) the degree of the $\deg(f)=\max(\deg(P), \deg(\phi))$, with $\deg(P)$ being the degree of the polynomial. We will call the pseudo polynomial \textit{dominant} if $\deg(f)=\deg(\phi)$, while reserving the name \textit{non-dominant} for the case $\deg(f)=\deg(P)$. Bergelson \textit{et.al.} \cite{TichyBergelson} proved, among other results concerning pseudo-polynomials, that  for a given pseudo polynomial $f$, the sequence $(f(p)_p)$ is uniformly distributed modulo 1. This motivated Madritsch and Tichy to investigate  Diophantine properties of pseudo polynomials. For the analogous Davenport's small fractional part problem, they obtained the following result: given any $\xi\in [0,1]$ we have
\begin{equation}\label{TichyMadritsch-resultovernaturals}
\min_{1\leq n\leq}\Vert \xi \lfloor f(p)\rfloor \Vert \ll_{f,\epsilon} X^{-\rho_1(f)+\epsilon},
\end{equation}
where the implicit constant depends upon $f$ and $\epsilon$, see \cite{Tichy2019}. For the analogous localized problem along primes, they established the following result
\begin{thm}[Madritsch-Tichy, \cite{Tichy2019}] \label{MTtheorem}
Given a pseudo polynomial $f$, any real $\xi$ and $X\in \mathbb{N}$ sufficiently large, there exists an exponent $\rho_1>0$ such that 
\begin{align}\label{MTexponent}
\min_{\substack{2\leq p\leq X\\ p \text{ prime}}}\Vert \xi \lfloor f(p)\rfloor \Vert \ll_f X^{-\rho_2(f)+\epsilon}.
\end{align}
\end{thm}
The exponents $\rho_1(f), \rho_2(f)$  were explicitly given for any $f$. They depend on a rather technical expression in $\deg(\phi)$ and $\deg(P)$, with  distinctions depending upon $f$ being dominant or not. In particular, for the so called Piatetski-Shapiro sequence $(\lfloor n^c\rfloor+n^k)_{n\geq 1}$ with $c>1$ they obtained, taking $f(x)=x^c+x^k$ the following corollary
\begin{cor}[Madritsch-Tichy]
Let $\xi$ be real, let $c>1$ be non integral and $f(x)=x^c+x^k$, then, for any $\epsilon>0$ we have
\begin{equation}\label{MTpjatetski}
\min_{\substack{2\leq p\leq X\\ p \text{ prime}}}\Vert \xi \lfloor f(p)\rfloor \Vert \ll_{c,k, \epsilon} X^{-\rho_2(c, k)+\epsilon},
\end{equation}
where
\[
    \rho_2(c, k)= 
\begin{cases}
    \frac{1}{2(2^{\lceil c\rceil+1}-1)},& \text{if } c>k\\
    \frac{1}{4^{k-1}(k+2)},              & \text{if } c<k.
    \end{cases}
\]
\end{cor}
\subsection{The goal:} In this work we focus on improving upon these two results in the case $f$ is dominant. In particular, the exponent we obtain supersedes that in \cite{Tichy2019} whenever $deg(f)>4$, We are also able to obtain improvements in the non-dominant case for $f$ with further additional mild conditions. However, for the purpose of keeping the present document as short as possible, we decided to keep the two cases separated.
\subsection{Our results}
\begin{thm}\label{Ourtheorem}
Let $f$ be a dominant pseudo polynomial of degree $\theta>3$ and let  $\xi$ be a real number. Then, we have
\begin{align}\label{ourexponent}
\min_{\substack{2\leq p\leq X\\ p \text{ prime}}}\Vert \xi \lfloor f(p)\rfloor \Vert \ll_f X^{\rho_d+\epsilon},
\end{align}
where
\begin{align*}
\rho_d=\frac{1}{3}\frac{1}{8\theta^2+12\theta+10}.
\end{align*}
\end{thm}
From this we descend the following corollary
\begin{cor}
Let $\xi$ be real, let $k$ be a positive integer and let $c>3$ be a non integer real number with $c>k$. Let $f(x)=x^c+x^k$, then, for any $\epsilon>0$ we have
\begin{equation}\label{ourpjatetski}
\min_{\substack{2\leq p\leq N\\ p \text{ prime}}}\Vert \xi \lfloor f(p)\rfloor \Vert \ll_{c,k, \epsilon} N^{-\rho_d+\epsilon},
\end{equation}
where
$$\rho_d=\frac{1}{24c^2+36c+30}$$
\end{cor}

The improvements we obtain is a consequence of a combination of a relatively recent new derivative test (see Lemma \ref{HBlemma}), a more careful estimation of type II sums (see proof of Lemma \ref{T2estimate}) and some additional care in the second part of the argument used by Madritsch and Tichy (see \cref{lemmamultiplesofm}).

\subsection{Notation:} 
In the present work, the letter $p$ will always indicate a prime number. As usual, by $\Vert x\Vert$ we denote the distance of $x$ from the nearest integer. For dyadic sums as $\sum_{X/2< n\leq X}$ we will write $\sum_{n\sim X}$. We will write $d_m(n):=\sum_{n_1\dots n_m=n} 1$ and say  that a sequence is divisor bounded it if is bounded by $d_4$ (i.e $\vert a_n\vert \leq d_4(n)$). Finally, we will employ the standard notation $e(x):=e^{2\pi i x}$.

\subsection{Acknowledgements}
The author thanks Marc Technau and his advisor Prof. R. F. Tichy for various helpful discussions. The author was supported by FWF project I-3466.

\section{Preparatory lemmas}
In this section we state all the lemmas we will require for the proof of \cref{Ourtheorem}. The first two lemmas are standard. The first will be used to remove the floor function in the proof of \cref{Ourtheorem}.  \cref{HBlemma} is is realatively recent and will be our main tool in estimating exponential sums. Finally, \cref{primesums} and \cref{lemmamultiplesofm} will be the two key ingredients for the proof. \cref{lemmamultiplesofm} may be regarded as a substitute of the well known principle for Weyl sums: one either has "good estimates" or the coefficients of the polynomial are well approximated by ratios with a common denominator $q$ of appropriate size.
\begin{lem}[Vaaler, see e.g. \cite{Vaaler}]\label{Vaaler}
Let $I$ be some interval modulo one and let $1_I$ denote its indicator function. Then for every positive integer $H$ there are coefficients $c_h=c_h(I, H)$, with 
\begin{align*}
c_0:=\meas(I \cap [0,1)), \quad \vert c_h\vert \leq \frac{1}{\vert h\vert +1},
\end{align*}
such that the difference 
\begin{align*}
\Delta_{I, H}(t)=1_I(t)-\sum_{0 \leq \vert h\vert \leq H } c_h e(ht)
\end{align*}
satisfies 
\begin{align*}
\vert \Delta_{I,H}(t)\vert \leq \frac{1}{2H	}\sum_{0\leq h\leq H} \left(1-\frac{\vert h\vert}{H}\right) e(ht).
\end{align*}
\end{lem}
Next we would require the following lemma from \cite{HBShapiro}
\begin{lem}[Heath-Brown, \cite{HBShapiro}]\label{HBVaughan}
Let $3\leq U<V<Z<X$ and suppose that $z$ is an half integer. Assume further that these variables satisfy $Z\geq 4U^2$, $X\geq 64Z^2U$, $V^3\geq 32X$. Let now $f$ be a function supported in $[X/2, X]$ with $\vert f(n)\vert \leq f_0$. Then
\begin{align*}
\Big\vert \sum_{n\sim X} \Lambda(n)f(n)\Big\vert \ll f_0+K\log X+L\log^8 X,
\end{align*}
where 
\begin{align*}
K:=\max_{N}\sum_{m=1}^\infty d_3(m)\Big\vert \sum_{Z< n\leq N} f(mn)\Big\vert,
\end{align*}
and
\begin{align*}
L:=\sup \sum_{m=1}^\infty d_4(m)\Big\vert \sum_{U<n<V} g(n)f(mn)\Big\vert
\end{align*}
where the supremum is taken on all arithmetic functions fulfilling $\vert g(n)\vert \leq d_3(n)$. 
\end{lem}
\begin{lem}[see e.g \cite{BakerIneq}]\label{montgomery}
Let $x_n$ be a sequence of reals with $\Vert x_n\Vert \geq \frac{1}{M}$ for $1\leq n \leq N$. Then
\begin{align*}
\sum_{m\leq M} \Big\vert \sum_{n\leq N} e\left(mx_n\right)\Big\vert>\frac{N}{6}.
\end{align*}
\end{lem}

The following recent $k$-derivative test of Heath-Brown (see Theorem 1  in \cite{HBKDT}) will be crucial. The version stated below matches that appearing in \cite{Kumchev2019}.
\begin{lem}[Heath-Brown]\label{HBlemma}
Let $F$ and $X$ be large and assume $X\leq Y\leq 2X$. Let $k\geq 3$ be an integer, and $f: [X,Y]\to \mathbb{R}$ be a $k$-times continuously derivable function which satisfies the following 
\begin{align}\label{HBlemmaeq1}
FX^{-k}\ll \Big\vert f^{(k)}(x)\Big\vert \ll FX^{-k} \qquad x\in (X,X_1].
\end{align}
Then we have the estimate
\begin{equation}\label{Hblemmaeq2}
\sum_{X <n\leq X_1} e(f(n)) \ll X^{1+\epsilon} \times \left[\left(FX^{-k}\right)^\frac{1}{k(k-1)}+X^{-\frac{1}{k(k-1)}}+F^{-\frac{2}{k^2(k-1)}}\right].
\end{equation}
where the implicit constant above may depend upon those in \cref{HBlemmaeq1} and the level of differentiation $k$.
\end{lem}
\begin{proof}
If $Y-X\gg X$, then the results follows plainly from \cite{HBKDT} after shifting the function by $X$. If $Y-X$ is of lower order, then $2X-Y\gg X$, so applying the first case to sums over intervals $X\leq n\leq 2X$ and $Y\leq n\leq 2X$ and using the triangle inequality, we recover \cref{Hblemmaeq2}.
\end{proof}

Finally, we need the following estimate for exponential sums over primes.
\begin{lem}[Prime exponential sums]\label{primesums}
Let $f$ be a dominant pseudo polynomial of degree $\theta>3$. Let $X^{-\frac{2}{3}\theta}\ll y\ll X^{\rho(1-\rho)}$, where
\begin{align}\label{rhodefinition}
\rho=\frac{1}{8\theta^2+12\theta+10}.
\end{align}
Then we have
\begin{align*}
\sum_{p\leq X} e\left(y f(p)\right)\ll X^{1-\rho+\epsilon}.
\end{align*}
\end{lem}
\begin{proof}
Consider the sum 
\begin{align*}
\sum_{n\leq X} \Lambda(n)e\left(yf(n)\right).
\end{align*}
The contribution from powers $p^k$, $k\geq 2$ is $\ll X^{\frac{1}{2}}$, thus, by partial summation it is enough to show that the sum above is $\ll X^{1-\rho+\epsilon}$. Splitting now the summation range into dyadic intervals and estimating trivially on intervals of size $\ll X^{1-\rho}$, we may consider only sums 
\begin{align*}
\sum_{n\sim Y} \Lambda(n)e\left(y f(n)\right)
\end{align*}
where $X^{1-\rho}\ll Y\ll X$. At this point, we appeal to \cref{HBVaughan} with parameters $U=Y^{2\rho}$, $V=4Y^{\frac{1}{3}}$, and $Z$ to be the half integer nearest to $\frac{1}{9}Y^{\frac{1}{2}-\rho}$. Now our sum decomposes as
\begin{equation}\label{proof:lemmaprimesums:split}
\sum_{n \sim Y} \Lambda(n) e\left( y f(n) \right)\ll f_0 + K\log X+ L \log^8 X,
\end{equation}
where 
\begin{align*}
K=\sum_{m=1}^\infty  \sum_{\substack{Z<n\leq Y\\ nm\sim Y}}a_m e\left(yf(mn)\right)
\end{align*}
and
\begin{align*}
L:=\sum_{\frac{Y}{2V}\leq m<\frac{Y}{U}} \sum_{\substack{mn\sim Y\\ U< n<V}} a_mb_n e\left( y f(mn)\right),
\end{align*}
where $(a_m)_m$ and $(b_n)_n$ are divisor bounded sequences of complex numbers. Now the sum
\begin{align*}
K=\sum_{m=1}^\infty d_3(n)\Big\vert \sum_{\substack{Z<n\leq Y\\ nm\sim Y}}e\left(yf(mn)\right)\Big\vert=\sum_{m=1}^\infty \sum_{\substack{Z<n\leq Y}} a_m e\left(yf(mn)\right),
\end{align*}
with $(a_m)_m\subset \mathbb{C}$, can be further decomposed  into $\ll \log X$ sub-sums of shape
\begin{align}\label{equation1:proofoflemmaexpsumsamongprimes:dominant:prime}
\sum_{m=1}^M\sum_{\substack{n\sim N\\ mn\sim Y}} a_m e\left(yf(mn)\right),
\end{align}
where $M\leq Y/Z\ll Y^{\frac{1}{2}+\rho}$. Proceeding similarly for the sum $L$, decomposing this into $\log^2 X$ sub-sums of shape
\begin{align}\label{equation2:proofoflemmaexpsums:dominant:prime}
\sum_{m \sim M}\sum_{\substack{n\sim N\\ mn \sim Y}} a_m b_n e\left(yf(mn)\right),
\end{align}
where $Y^{2\rho}=U<M<V=Y^{\frac{1}{3}}.$ 
The sums \cref{equation1:proofoflemmaexpsumsamongprimes:dominant:prime} and \cref{equation2:proofoflemmaexpsums:dominant:prime} can be estimated\footnote{The careful reader will notice that here the range for $y$ is given in terms of $X$, while the length of the sum is $Y$. However, the range in \cref{primesums} is thinner than that considered for \cref{TypeIestimate} and \cref{T2estimate}. As $Y\gg X^{1-\rho}$, one sees that the type I and type II estimates are applicable.} using \cref{T1estimate} and \cref{T2estimate}. Hence, by \cref{proof:lemmaprimesums:split} 
\begin{align*}
\sum_{n\sim Y} \Lambda(n)e\left(yf(mn)\right)\ll X^{1-\rho+\epsilon}.
\end{align*}

\begin{lem}\label{lemmamultiplesofm} Let $f$ be a dominant pseudo polynomial and let  $2\leq m\leq X^{\tilde{\rho}}$, where $\tilde{\rho}< \frac{1}{3}\rho$ (with $\rho$ as in \cref{rhodefinition}). Then, for $X$ sufficiently large there is a prime $p\ll X^{\frac{1}{3}+\epsilon}$ such that $\lfloor f(p)\rfloor$ is divisible by $m$. 
\end{lem}
\begin{proof}
	[Proof of \cref{lemmamultiplesofm}]
	We follow essentially \cite{Tichy2019}. Let $J$ be the interval $\left[0, \frac{1}{m}\right)$. Then $m\vert \lfloor f(p)\rfloor\Leftrightarrow \frac{f(p)}{m}\in [0, 1/m)$ modulo 1. 
	We want to show that
	\begin{align}\label{Sa}
	S_A(Y):=\# \{ p\leq Y: p \text{ prime} \text{  and  } m\vert \lfloor f(p)\rfloor\}=\sum_{\substack{p\leq Y\\ m\vert \lfloor f(p)\rfloor}} 1 >0
	\end{align}
	for $Y$ large enough, which will ensure the existence of a prime of the desired type. To this end we compare \cref{Sa} with the sum over all primes $\leq Y$, which we denote by $S_B$.
	\begin{align}\label{Sadiff}
	S_{A}-\frac{1}{m}S_{B}&=\sum_{p\leq Y}\left(1_J \left(\frac{f(p)}{m}\right)-\frac{1}{m}\right)\\
	&\ll \frac{1}{\log P}\max_{P\leq Y}\Big\vert \sum_{n\leq P}\Lambda(n)\left(1_{J}\left(f(n)/m\right)-\frac{1}{m}\right)\Big\vert +O\left(\sqrt{Y}\right).\nonumber
	\end{align}
	Set now $Y=X^{\frac{1}{3}+\epsilon}$, where $\epsilon>0$ small. Using  \cref{Vaaler} to smooth the characteristic function of $J$ we have
	\begin{align}\label{Sadiffexpsumestimate}
	\Big\vert \sum_{n\leq Y}\Lambda(n)\left(1_{J}\left(\frac{f(n)}{m}\right)-\frac{1}{m}\right)\Big\vert&\ll \sum_{1\leq h \leq H} \frac{1}{h}\sum_{n\leq Y}\Lambda(n)e\left(\frac{h}{m}f(n)\right)\\ \nonumber
	&+\frac{1}{H+1}\sum_{h\leq H} \left(1-\frac{\vert h\vert}{H+1}\right) \sum_{n\leq Y} \Lambda(n) e\left(\frac{h}{m}f(n)\right)\\\nonumber
	&\ll X^{(\frac{1}{3}+\epsilon)(1-\rho+\epsilon)},
	\end{align}
	where we have taken $H=Y^{\rho}$ and applied \cref{primesums}. Now by the prime number theorem and $m\leq X^{\tilde{\rho}}$ we have
	\begin{equation}\label{S_Blowerbound}
	\frac{1}{m}S_B\left(X^{\frac{1}{3}+\epsilon}\right)\gg\frac{X^{\frac{1}{3}+\epsilon}}{m \log X}\gg \frac{X^{\frac{1}{3}-\tilde{\rho}+\epsilon}}{\log X}.
	\end{equation}
	As $\tilde{\rho}< \frac{\rho}{3}$, the lower bound \cref{S_Blowerbound} dominates the upper bound \cref{Sadiffexpsumestimate} we conclude that $S_A(X^{\frac{1}{3}+\epsilon})\geq 1$.
\end{proof}

\section{Proof of \cref{Ourtheorem}}\label{proofofthetheorem}
We proceed by contradiction. Assume that   
\begin{equation}\label{contradeq}
\min_{2\leq p\leq X} \Vert \xi \lfloor f(p)\rfloor\Vert\geq X^{-\tilde{\rho}}
\end{equation}
for $0<\tilde{\rho}<\rho_d$, where $\rho_d:=\frac{1}{3}\rho$, and $\rho$ given as in the statement of \cref{primesums}. Set now $M:=\lfloor X^{\tilde{\rho}}\rfloor$. By \cref{contradeq} and \cref{montgomery} we are given an $m\leq M$ with
\begin{equation}\label{lowerbound}
\Big\vert \sum_{p\leq X} e\left(m\xi \lfloor f(p)\rfloor\right)\Big\vert\gg X^{1-\tilde{\rho}}.
\end{equation}
We proceed now in proving an upper bound for the left hand side of \cref{lowerbound}. To this end, we adopt the strategy used in \cite{Tichy2019} i.e to use digital expansion to remove the floor function.
Let now $q\geq 2$ be an integer parameter that will be specified later. Let  $0 \leq d \leq q-1$ and let $I_d$ be the interval (to be read modulo 1) $[d/q, (d+1)/q).$ Write $f(n)=\lfloor f(p)\rfloor + \{ f(p)\}$. Then $\{f(p)\}\in I_d$  precisely when $\{f(p)\}=\frac{d}{q}+\frac{\tau}{q}$ with $\tau\in [0,1)$, hence
\begin{align*}
e\left(m\xi \lfloor f(p)\rfloor\right)=e\left(m\xi f(p)-m\xi\frac{d}{q}\right)\left(1+O\left(\frac{m}{q}\right)\right).
\end{align*}
Then
\begin{equation}\label{proofeq1}
 \sum_{p\leq X} e\biggl(m\xi \lfloor f(p)\rfloor\biggr)= \sum_{d=0}^{q-1} \sum_{p\leq X} e\biggl(m\xi f(p)-m\xi\frac{d}{q}\biggr)1_{I_d}(f(p))+O\biggl(\frac{Xm}{q}\biggr).
\end{equation}
Now, an application of \cref{Vaaler} reduces the estimation of the sums on the right hand side of \cref{proofeq1} to those of the following three sums
\begin{enumerate}
\item \(\displaystyle \frac{1}{q}\Big\vert \sum_{p\leq X} e\left(m\xi f(p)\right)\Big\vert\),
\item \(\displaystyle \sum_{0\leq \vert h\vert\leq H} \frac{1}{h} \Big \vert \sum_{p\leq X} e\left((m\xi+h)f(p)\right)\Big\vert\),
\item \(\displaystyle \frac{1}{H+1}\sum_{\vert h\vert \leq H} \left(1-\frac{\vert h\vert}{H+1}\right)\Big\vert \sum_{p\leq X} e\left(hf(p)\right)\Big\vert, \)\end{enumerate}
where $H=X^{\rho+\epsilon}$. We assume now that $\Vert m\xi+h\Vert\gg X^{-2\theta/3+1/4}$. Now, an application of \cref{primesums} yields the bound 
\begin{align*}
(\ref{proofeq1})\ll qX^{1-\rho+\epsilon}\times \left( \frac{1}{q}+ \sum_{0<\vert h\vert \leq H} \frac{1}{\vert h \vert }+\frac{1}{H+1}\sum_{\vert h\vert \leq H} \left(1-\frac{h}{H+1}\right)\right)+ O\left(\frac{Xm}{q}\right).
\end{align*}
At this point we have to balance the two terms. The optimal selection is $q=\lfloor \sqrt{mX^\rho}\rfloor$. This gives us 
\begin{align*}
(\ref{proofeq1})\ll X^{1-\frac{\rho}{2}+\epsilon}m^{\frac{1}{2}}\ll X^{1-\frac{\rho}{2}+\frac{\tilde{\rho}}{2}+\epsilon},
\end{align*}
which together with \cref{lowerbound} gives us a contradiction whenever 
\begin{align}\label{condition1}
\tilde{\rho}<\rho_d=\frac{\rho}{3}.
\end{align}
We still have to deal with the possibility that $\Vert m\xi+h\Vert\ll X^{-2\theta/3+1/4}$. In this case, if $m=1$ then we have
\begin{align*}
\Vert \xi\lfloor f(p)\rfloor \Vert \ll X^{-\theta/3+1/4}\ll X^{-1/12},
\end{align*}
contradicting \cref{contradeq}. If otherwise $m>1$ then  \cref{lemmamultiplesofm} provides us a prime $p\ll_{k,\theta} X^{1/3+\epsilon}$ such that $m\vert \lfloor f(p)\rfloor$, whence, for this specific $p$ we have
\begin{align*}
\Vert \xi \lfloor f(p)\rfloor \Vert\ll \Vert m\xi\Vert \frac{\lfloor f(p)\rfloor}{m}\ll X^{-\frac{2}{3}\theta+1/4} X^{(\frac{1}{3}+\epsilon)\theta}\ll X^{-\frac{\theta}{3}+\theta\epsilon+1/4}\ll X^{-\frac{1}{20}},
\end{align*}
which clearly contradicts \cref{contradeq}. This completes the proof.
\end{proof}

\section{Exponential sums estimates for  and Type II sums.}
The estimation ef exponential sums done by Madritsch and Tichy in \cite{Tichy2019} relied on Weyl-Van Der Corput differencing together with some derivative estimate (see \cite{TichyBergelson} and Section 4 of \cite{Tichy2019}). The use of Weyl differencing and classical derivative tests leads to an exponent factor of type $1-\frac{1}{2^k}$. This is rather unpleasant in case the degree of the pseudo polynomial (and  therefore the required differentiation level) is high. We remedy this  developing our estimates from \cref{HBlemma}. As mentioned in the introduction, during the preparation of the present note, we came across the paper of Kumchev and Petrov \cite{Kumchev2019}, where exponential sums of similar shape were estimated. From this paper, we borrowed the application of the Van Der Corput inequality in Type II sums, which leads to a slightly better saving than the one we originally obtained. The saving exponent we obtain in \cref{TypeIestimate} and \cref{T2estimate} is almost equivalent to the saving obtained in \cite{Kumchev2019}. 
The first two terms cannot be improved without a substantial change in the estimation method, while one could optimize the constant term in the denominator. However, we think that the exponent we provided is more transparent. The proof of this estimate is a direct application of \cref{HBlemma} and the standard repertory concerning derivative tests, see \cite{Kolesnik}.
\begin{lem}[Type I]\label{TypeIestimate}
Let $f$ be a dominant pseudo polynomial of degree $\theta>3$ and let $(a_n)_{n\in \mathbb{N}}$ be a divisors bounded sequence. Set 
\begin{align*}
\rho=\frac{1}{8\theta^2+12\theta +10},
\end{align*}
and let $X^{-\frac{2}{3}\theta}\ll y\ll X^{\rho}$,  $M\ll X^{{1/2}+\rho}.$ Then we have
\begin{equation}\label{T1estimate}
\sum_{m \leq M} \sum_{mn \sim X} a_m e(yf(mn))\ll X^{1-\rho+\epsilon}.
\end{equation}
\end{lem}
\begin{proof}
We denote the left side of \cref{T1estimate} by $S$ and we write $X_m$ for $X/m$. Since the sequence $a_m$ is divisor bounded, we have
\begin{align*}\label{Type1}
S\ll X^\epsilon\times \sum_{m\leq M}\Big\vert \sum_{n\sim X_m} e(yf(mn))\Big\vert.
\end{align*}
Let us write $X^\alpha = y X^\theta$. Notice that by our assumptions we have $1<\theta/3\leq \alpha\leq \theta+\rho$. Now, on the inner summation range we have 
\begin{align*}
X^{\alpha}X_m^{-k}\ll_{k, \theta} y\partial_{n}^k f(mn)\ll_{k,\theta}  X^{\alpha}X_m^{-k},
\end{align*}
where the implicit constant depends upon $\theta$ and $k$ but not on $X$. We will estimate the inner sum over $n$ using an appropriate derivative test.
\\
Assume for the moment $\alpha> 1+2\rho$. We select the degree of differentiation in such a way, the first term inside the brackets in \cref{Hblemmaeq2} is $ X^{-2}\ll$ and $\ll X^{-1}$. Because of the condition $m\ll X^{\frac{1}{2}+\rho}$, this means $1+\frac{\alpha}{1/2-\rho}\leq k< 2 + \frac{\alpha}{1/2-\rho}$. We select $k=\lceil \frac{\alpha}{1/2-\rho}\rceil+1$. An application of \cref{HBlemma} (which is applicable in the given range, as $k$ would be larger or equal than 3) leads to 
\begin{align*}
\sum_{m\leq M}\sum_{n\sim X/m} e\left(yf(mn)\right)&\ll \sum_{m\leq M} \left(X_m^{1-\frac{1}{k(k-1)}+\epsilon}+X_m^{1+\epsilon}X^{-\frac{2\alpha}{k^2(k-1)}}\right)\\
&\ll X^{1-\frac{1/2-\rho}{k(k-1)}+\epsilon}+X^{1-\frac{2\alpha}{k^2(k-1)}+\epsilon}\\
&\ll X^{1-\frac{1/2-\rho}{k(k-1)}+\epsilon}\\
&\ll X^{1-\rho+\epsilon},
\end{align*}
where we used that $\alpha>1+2\rho$ and the definition of $\rho$.
\\
If otherwise $\alpha\in [1, 1+2\rho)$ then we evaluate the inner sum using Lemma  2.9 of \cite{Kolesnik}. Summing over $m$ gives us
\begin{align*}
\sum_{m\leq M} X_mX^{-\alpha}+ X_m^\frac{1}{2}X^{\frac{\alpha}{6}}\ll X^{1-\alpha+\epsilon}+M^\frac{1}{2}X^{\frac{\alpha}{6}+\frac{1}{2}}\ll X^{1-\alpha+\epsilon}+X^{\frac{3}{4}+\frac{\alpha}{6}+\frac{\rho}{2}}\ll X^{1-\rho+\epsilon},
\end{align*}
where the last equation is a consequence of $\rho<\frac{1}{72}$.
\end{proof}

\begin{lem}[Type II]\label{T2estimate}
Let $f$ be a dominant pseudo polynomial of degree $\theta>3$. Let $(a_m)_{m\in \mathbb{N}}$ and $(b_n)_{n\in\mathbb{N}}$ be divisors bounded sequences. Let $X^{2\rho}\ll M\ll X^{\frac{1}{3}}$, $\rho$ as in \cref{TypeIestimate} and assume that
\begin{align*}
X^{-\frac{2}{3}\theta}\ll y\ll X^{\rho +\epsilon}.
\end{align*}
Then we have
\begin{align*}
\sum_{m\sim M}\sum_{\substack{n\sim N\\ mn\sim X}} a_mb_n e(yf(mn))\ll X^{1-\rho+\epsilon}.
\end{align*}
\end{lem}

\begin{proof}
We shall assume that neither of the two coefficient sequences $(a_m)_{m\in\mathbb{N}}$ and $(b_n)_{n\in\mathbb{N}}$ is identically zero in the relevant ranges, otherwise the claimed bound holds trivially. Now, an application of Cauchy's inequality gives us 
\begin{equation}\label{eq1pflemmaT2estimatelarge}
\vert S\vert^2 \le \left(\sum_{n\sim N} \vert b_n\vert^2 \right) \sum_{n\sim N} \Big\vert \sum_{\substack{m \sim M\\ nm \sim X}} a_m e(yf(mn))\Big\vert^2.
\end{equation}
We estimate the inner sum via the Van Der Corput lemma (see e.g. Chapter 2 of \cite{Montgomery10lect}) with $H=X^\tau$ for some small $\tau$ that will be specified later. This yields 
\begin{align*}
\Big\vert \sum_{\substack{m \sim M\\ nm \sim X}} a_m e(yf(mn))\Big\vert^2&\ll \frac{M^2}{H}\log^3 M \\
&+ \frac{M}{H}\sum_{0<\vert h\vert \leq H} \left(1-\frac{\vert h\vert}{H+1}\right)\ \sum_{\substack{m\sim M\\ nm\sim X\\ (m+h)n\sim X}} a_m^\star a_{m+h} e\left(y(f[m(n+h)]-f[mn]\right),
\end{align*}
where the first term arise from collecting the terms at $h=0$. At this point we insert this in \cref{eq1pflemmaT2estimatelarge}, then we sum over $n$, use the fact that the sequences are divisor bounded, and change summation order in the sum above, obtaining
\begin{align}\label{S2norm}
S^2 &\ll \left(\sum_{n\sim N} \vert b_n\vert^2\right)\\
&\quad\times \Biggl(\frac{M^2N}{H}\log^3 M +\frac{M^{1+\epsilon}}{H}\sum_{0<\vert h\vert \leq H} \left(1-\frac{\vert h\vert}{H+1}\right) \sum_{m\sim M}  \Big\vert \sum_{\substack{n\sim N\\ mn\sim X\\ (m+h)n\sim X}} e\left(F_h(mn)\right)\Big\vert\Biggr),\nonumber
\end{align}
where we used the abbreviation $F_h(mn)=y\left(f((m+h)n)-f(mn)\right)$ and $\epsilon>0$. Notice now that we can suppress the third condition in the innermost sum above at the price of an additive error about $N^2MH$ in the bound for $S^2$, which is acceptable. Now we estimate the inner sum using a suitable derivative test. Write again for shortness $X^\alpha=yX^\theta$ and $X_m=X/m$. Notice that since $\tau$  will be chosen smaller that $2\rho$ (whence $H=o(M)$), we have
\begin{align*}
\partial_n^k F_h(mn)=y\left((m+h)^\theta-m^\theta\right)N^{\theta-k}\asymp yX^{\theta-1} \vert h\vert N^{1-k}=X^{\alpha-1}\vert h\vert X_m^{1-k}
\end{align*}
for $n$ in the given range. Now we must distinguish between some cases:
\\
Assume that $\alpha\geq 121/60$. We apply \cref{HBlemma} with $F=\frac{y\vert h\vert X^\theta}{m}$, and select the differentiation level\footnote{See discussion in the proof of the previous lemma}  to be $k=\lceil \frac{3}{2}\left(\alpha-1+\tau\right)\rceil+2$. We have 
\begin{align*}
\Big\vert \sum_{\substack{n\sim N\\ mn\sim X}} e\left(F_h(mn)\right)\Big\vert &\ll X_m^{1+\epsilon}\left(X_m^{-\frac{1}{k(k-1)}}+F^{-\frac{2}{k^2(k-1)}}\right)\\
&\ll X_m^{1+\epsilon}\left[X_m^{-\frac{1}{k(k-1)}}+\left(\frac{\vert h \vert X^\alpha}{m}\right)^{-\frac{2}{k^2(k-1)}}\right].
\end{align*}
Thus, summing over $m\sim M$ we obtain
\begin{align*}
S^2& \ll \left(\sum_{n\sim N} \vert b_n\vert^2\right)\times \left[\frac{M^2N}{H}\log^3 M +\frac{M^{1+\epsilon}}{H}\sum_{1\leq h\leq H} \left(X^{1-\frac{2}{3k(k-1)}}+\vert h\vert^{-\frac{2}{k(k-1)}}X^{1-\frac{2\alpha}{k^2(k-1)}+\epsilon}\right)\right]\\
& \ll\frac{X^{2+\epsilon}}{H}+X^{2-\frac{2}{3k(k-1)}+\epsilon}+H^{\frac{2}{k^2(k-1)}}X^{2-\frac{2\alpha-2/3}{k^2 (k-1)}+\epsilon}	
\end{align*}
At this point we select $\tau=2\rho-\epsilon$. Because of $\alpha \leq \theta+\rho$ and our selection of $k$, we see that the first term dominates the second. We also notice that the second term dominates the third precisely when 
\begin{align*}
\frac{2}{3k(k-1)}\geq \frac{\left(2\rho-2\alpha+\frac{2}{3}\right)}{k^2(k-1)}
\end{align*}
which happens when $\alpha>2+4\rho$. Because of the fact $\rho<1/72$, the theorem is proved for $\alpha$ in the given range.
\\\
If otherwise we are in the range $1<\alpha<\frac{166}{60}$ then we can estimate  \cref{S2norm} using Theorem 2.9 of \cite{Kolesnik}. We have $\partial_n^{(3)} F_h(mn) \asymp  X^{\alpha-1} \vert h\vert X_m^{-2}$. Then we have
\begin{align*}
S^2&\ll N^{1+\epsilon}\biggl(\frac{M^2N}{H}\log^3 M +\frac{M^{1+\epsilon}}{H}\sum_{0<\vert h\vert \leq H} \sum_{m\sim M}  \Big\vert \sum_{\substack{n\sim N\\ mn\sim X}} e\left(F_h(mn)\right)\Big\vert\biggr)\\
&\ll N^{1+\epsilon} \biggl( \frac{M^2N}{H}\log^3 M +\frac{M^{1+\epsilon}}{H}\sum_{0<\vert h\vert \leq H} \sum_{m\sim M} \left(X^{\frac{\alpha-1}{6}}\vert h\vert^{\frac{1}{6}}N^{\frac{1}{2}}+\frac{X^{1-\alpha}}{\vert h\vert} \right)\biggr)\\
&\ll \frac{X^{2+\epsilon}}{H}+H^{\frac{1}{6}}X^{\frac{3}{2}+\frac{\alpha}{6}+\epsilon}+\frac{M}{H}X^{2-\alpha+\epsilon}.
\end{align*}
Selecting $\tau=2\rho-\epsilon$ as above and recalling $M\ll X^{1/3}$, we conclude once again
\begin{align*}
S\ll X^{1-\rho+\epsilon}.
\end{align*}
\end{proof}

\bibliographystyle{plainnat}
\bibliography{PPatprimes}

\end{document}